\documentclass[12pt]{amsart}
\usepackage{color}
\usepackage{axodraw4j}
\usepackage{pstricks}
\usepackage{amssymb}
\usepackage{mathtools}
\usepackage{txfonts}
\usepackage{eucal}
\usepackage{extarrows}
\usepackage{wasysym}
\usepackage[all]{xy}

\usepackage{mathrsfs} 

\usepackage{caption}
\usepackage{float} 
\usepackage{amsmath}
\usepackage{tikz}
\usepackage{cite}
\usepackage{graphicx}
\usepackage{float}

\usepackage{xspace}
\usepackage[a4paper,body={16.3cm,22.8cm},centering]{geometry}
\usepackage[colorlinks=true,pagebackref=true]{hyperref} 
\hypersetup{urlcolor=blue, citecolor=red , linkcolor= blue}
\usepackage{shuffle}
\font \eightrm=cmr8

\newcommand{\nc}{\newcommand}
\nc\smsc{0.8}

\nc\delete[1]{}
\nc{\mlabel}[1]{\label{#1}}  
\nc{\mcite}[1]{\cite{#1}}  
\nc{\mref}[1]{\ref{#1}}  
\nc{\mbibitem}[1]{\bibitem{#1}} 

\delete{
\nc{\mlabel}[1]{\label{#1}  
{\hfill \hspace{1cm}{\small\tt{{\ }\hfill(#1)}}}}
\nc{\mcite}[1]{\cite{#1}{\small{\tt{{\ }(#1)}}}}  
\nc{\mref}[1]{\ref{#1}{{\tt{{\ }(#1)}}}}  
\nc{\mbibitem}[1]{\bibitem[\bf #1]{#1}} 
}
\delete{
\nc{\mop}[1]{\mathop{\hbox {\rm #1} }}
\nc{\smop}[1]{\mathop{\hbox {\eightrm #1} }}
\nc{\mopl}[1]{\mathop{\hbox {\rm #1} }\limits}
\nc{\smopl}[1]{\mathop{\hbox {\eightrm #1} }\limits}
\def \restr#1{\mathstrut_{\textstyle |}\raise-8pt\hbox{$\scriptstyle #1$}}
\def \srestr#1{\mathstrut_{\scriptstyle |}\hbox to
  -1.5pt{}\raise-4pt\hbox{$\scriptscriptstyle #1$}}
\nc{\wt}{\widetilde}
\nc{\wh}{\widehat}

}

\newtheorem{theorem}{Theorem}[section]
\newtheorem{definition}{Definition}[section]

\newtheorem{proposition}{Proposition}[section]
\newtheorem{lemma}{Lemma}[section]
\newtheorem{remark}{Remark}[section]
\newtheorem{example}{Example}[section]
\newtheorem{examples}{Examples}[section]
\numberwithin{equation}{section}
\newcommand\alphlist{a,b,c,d,e,f,g,h,i,j,k,l,m,n,o,p,q,r,s,t,u,v,w,x,y,z}
\newcommand\Alphlist{A,B,C,D,E,F,G,H,I,J,K,L,M,N,O,P,Q,R,S,T,U,V,W,X,Y,Z}
\newcommand\getcmds[3]{\expandafter\newcommand\csname #2#1\endcsname{#3{#1}}}
\makeatletter
\@for\x:=\alphlist\do{\expandafter\getcmds\expandafter{\x}{frak}{\mathfrak}}
\@for\x:=\Alphlist\do{\expandafter\getcmds\expandafter{\x}{frak}{\mathfrak}}
\makeatother

\nc{\bfk}{{\bf k}}
\nc{\sha}{\shuffle}

\nc{\id}{\mathrm{id}}
\nc{\Id}{\mathrm{Id}}
\nc{\lbar}[1]{\overline{#1}}
\nc{\ot}{\otimes}
\nc{\dep}{\mathrm{dep}}
\nc{\ver}{\mathrm{ver}}

\nc{\tred}[1]{\textcolor{red}{#1}} \nc{\tgreen}[1]{\textcolor{green}{#1}}
\nc{\tblue}[1]{\textcolor{blue}{#1}} \nc{\tpurple}[1]{\textcolor{purple}{#1}}
\nc{\tcyan}[1]{\textcolor{cyan}{#1}} 
\nc{\tblk}[1]{\textcolor{black}{#1}}

\nc{\li}[1]{\tpurple{\underline{Li:}#1 }}
\nc{\liadd}[1]{\tpurple{#1}}
\nc{\xing}[1]{\tblue{\underline{Xing:}#1 }}
\nc{\yuan}[1]{\tred{\underline{Yuan:}#1 }}
\nc{\markus}[1]{\tred{\underline{Markus:} #1}}
\nc{\dominique}[1]{\tpurple{\underline{Dominique: }#1 }}
\long\def\ignore#1{}

\usetikzlibrary{calc,shapes.geometric}
\tikzset{
baseon/.style={baseline={($(#1)+(0,-0.58ex)$)}},
baseon/.default=current bounding box.center,
every picture/.style=baseon,
lst/.style={},
dst/.style={circle,inner sep=1pt,outer sep=0pt,fill,draw,dst2},
dst2/.style={fill=white},
ddst/.style={diamond,draw,inner sep=1pt},
eest/.style={ellipse,draw,inner sep=1pt,minimum size=2ex},
}

\def\zzz#1`#2...#3`#4...#5`#6@{%
--++(#1)
node[dst,label={#5:$#6$},name=#2]{}
node[midway,\hbox{Aut}o,#3]{$#4$}
}
\def\ddd#1`#2`#3@{+(#1)node[ddst,name=#2]{$#3$}}
\def\eee#1`#2`#3@{+(#1)node[eest,name=#2]{$#3$}}
\def\xxx#1`#2@{node[midway,\hbox{Aut}o,inner sep=1pt,#1]{$#2$}}
\def\pp#1`#2`#3@{node[dst,label={#2:$#3$},pos=#1]{}}
\def\oo#1`#2`#3@{\path (o) node[dst,label={#2:$#3$},name=o,#1]{};}
\def\eoo#1`#2@{\node[eest,name=o,#1] at (o) {$#2$};}

\newif\ifshowjdq
\showjdqtrue
\newcommand\setXXclip[3]{%
\def\XXheight{#1}\def\XXdepth{#2}\def\XXwidth{#3}}
\setXXclip{1}{-0.5}{1.1}

\makeatletter
\newcommand\simra{\mathrel{\mathpalette\@verra\sim}}
\def\@verra#1#2{\lower.5\p@\vbox{\lineskiplimit\maxdimen \lineskip-.5\p@
\ialign{$\m@th#1\hfil##\hfil$\crcr#2\crcr\rightarrow\crcr}}}
\makeatother

\nc{\dnx}{\Delta_n A} \nc{\dx}{\Delta A} \nc{\dgp}{{\rm deg_{P}}}
\nc{\dgt}{{\rm deg_{T}}} \nc{\dg}{{\rm deg}} \nc{\ida}{ID($A$)} \nc{\tu}{\tilde{u}} \nc{\tv}{\tilde{v}}
\nc{\nr}{\calr_n} \nc{\nz}{\calz_n} \nc{\fun}{\cala_{n,d}}
 \nc{\fbase}{\calb} \nc{\LF}{\mathrm{RF}} \nc{\FFA}{\mathrm{LF}} \nc{\irr}{\mathrm{Irr}}
 \nc{\result}{\bfk\mathrm{Irr}(S_n)}  \nc{\I}{I_{\mathrm{ID},n}^0}
 \nc{\nrs}{\calr_n^\star} \nc{\ii}{\mathrm{I}} \nc{\iii}{\mathrm{II}}
\nc{\intl}{{\rm int}}\nc{\ws}[1]{{#1}}\nc{\deleted}[1]{\delete{#1}}\nc{\plas}{placements\xspace}

\nc{\bim}[1]{#1}  \nc{\shaop}{\sha_{\Omega}^{+}}  \nc{\shao}{\sha_{\Omega}}
\nc{\bbim}[2]{#1 #2} \nc{\bbbim}[2]{#1,\, #2} \nc{\RBF}{{\rm RBF}}
\nc{\frb}{F_{\RB}} \nc{\shaf}{\ssha_{\tiny{\Omega}}} \nc{\sham}{\blacklozenge_{\tiny{\Omega}}}
\nc{\lf}{\lfloor} \nc{\rf}{\rfloor} \nc{\shan}{\ssha_{\lambda}}
\nc{\rlex}{{\rm {lex}}} \nc{\bb}{\Box} \nc{\ra}{\rightarrow}
\nc{\e}{{\rm {e}}}
\nc{\DDF}{\mathrm{DD}(X,\,\Omega)}\nc{\DTF}{\mathrm{DT}(X,\,\Omega)} \nc{\DT}{\mathrm{DT}'(\Omega,\,V)}
\nc{\bra}{\mathrm{bra}} \nc{\bre}{\mathrm{bre}}
\nc{\dec}{\mathrm{dec}} \nc{\blacklozengew}{\blacklozenge_{w}}
\nc{\type}{\mathrm{type}}

\nc\calt{\cal{T}(X,\,\Omega)} \nc\caltn{\cal{T}_n(X,\,\Omega)}
\nc\calta{\cal{T}_0(X,\,\Omega)}
\nc\caltb{\cal{T}_1(X,\,\Omega)}
\nc\caltc{\cal{T}_2(X,\,\Omega)}
\nc\caltd{\cal{T}_3(X,\,\Omega)}
\nc\caltm{\cal{T}_m(X,\,\Omega)}
\nc\caltx{\cal{T}(X)}
\nc\calf{\cal{F}(X,\,\Omega)}
\nc\fram{\frak{M}(\Omega,\, X)}
\nc\shaw{\sha^{NC}_w(\Omega,\, X)}
\nc\dw{\blacklozenge_w} \nc\dl{\blacklozenge_\ell}
\nc\shal{\sha^{NC}_\ell(X,\, \Omega)} \nc\shav{\sha^{NC}_w(\Omega,\, V)} \nc\shat{\sha^{NC,1}_w(\Omega,\, T^{+}(V))}
\nc{\cfo}{\cal{F}(X,\,\Omega)}
\nc{\sh}{\rm{Sh}}
\nc{\lar}{\varinjlim}
\def\cxo#1#2;{\cal{#1}#2\XO}
\nc\lrf[2]{B_{#2}^+(#1)}
\nc{\fd}{\mathrm{\text{typed angularly decorated planar rooted trees}}}
\nc{\rb}{\mathrm{RBFWs}} \nc{\dfw}{\mathrm{DFW{(X)}}} \nc{\tfw}{\mathrm{TFW{(X)}}}
\nc{\tfv}{\mathrm{TFW{(V)}}}

\def\Ve#1,#2,#3;{\vee_{#1,\,(#2,\,#3)}}
\def\bigv#1;#2;#3;{\bigvee\nolimits_{#1}^{#2;\,#3}}
\nc\rjt[2]{\mathrel{\mathop{\longrightarrow}\limits^{#1\hfill}_{\hfill#2}}}
\nc{\pl}{\cal{PLF}}
\nc{\tr}{\cal{RTF}}
\nc{\im}{\mathrm{Im}}
\nc{\ff}{\cal{F}_\Omega}
\nc{\tm}{T_\Omega}
\nc{\calp}{\cal{P}}
\makeatletter
\nc\dd{\@ifnextchar'{\ddA}{\ddB}}
\def\ddA'#1;{\lhd'_{#1\,}}
\def\ddB#1;{\lhd_{#1\,}}
\nc{\pbt}{\mathrm{PBT}}
\nc{\ad}{\mathrm{ad}}

\makeatother

\begin{document}

\title[The topological quandles up to four elements]{The topological quandles up to four elements}
\thispagestyle{empty}
\author{Mohamed Ayadi}
\address{Laboratoire de Math\'ematiques Blaise Pascal,
CNRS--Universit\'e Clermont-Auvergne,
3 place Vasar\'ely, CS 60026,
F63178 Aubi\`ere, France, and University of Sfax, Faculty of Sciences of Sfax,
LAMHA, route de Soukra,
3038 Sfax, Tunisia.}
\email{mohamedayadi763763@gmail.com}
%
%

\tikzset{
			stdNode/.style={rounded corners, draw, align=right},
			greenRed/.style={stdNode, top color=green, bottom color=red},
			blueRed/.style={stdNode, top color=blue, bottom color=red}
		} 
		
	\begin{abstract}
The finite topological  quandles can be represented as
$n \times n$ matrices, recently defined by S. Nelson and C. Wong.
In this paper, we first study the finite topological quandles and we show how to use these matrices to distinguish all isomorphism classes of finite topological quandles for a given cardinality $n$. As an application, we classify finite topological quandles with up to 4 elements.
	\end{abstract}
	
\keywords{quandles, finite topological spaces .}
\subjclass[2020]{57K12, 16T05 .}
\maketitle
\tableofcontents
	\section{Introduction}
A quandle is a set $Q$ with a binary operation $\lhd : Q \times Q \longrightarrow Q$ satisfying the three
axioms
\begin{itemize}
    \item (i) for every $a \in Q$, we have $a \lhd a = a$,
    \item (ii) for every pair $a, b \in Q$ there is a unique $c \in Q$ such that $a = c \lhd b$, and
    \item (iii) for every $a, b, c \in Q$, we have $(a \lhd b) \lhd c = (a \lhd c) \lhd (b \lhd c)$.
    \\
\end{itemize}
As an example for $(G, \circ)$ a group and $\lhd: G\times G \longrightarrow G$ the operation defined by $x\lhd y=x\circ y\circ x^{-1}$, for all $x, y \in G$, then $Q$ is a quandle.\\

More on quandles can be found in \cite{L. Pedro and R. Dennis, B. Ho and S. Nelson, DN. Yetter}.\\

A quasi-poset is a pairs $(X, \le)$, where $X$ is a set and $\le$ a quasi-order on $X$, that is to say a transitive and reflexive relation on $X$.
Recall (see e.g. \cite{acg16, acg10}) that a topology on a finite set $X$ is given by the
family $\mathcal{T}$ of open subsets of $X$, subject to the three following axioms:
	\begin{itemize}
		\item $\hbox{\o}\in\mathcal{T}$, $X\in\mathcal{T}$,
		\item The union of (a finite number of) open subsets is an open subset,
		\item The intersection of a finite number of open subsets is an open subset.
	\end{itemize}


By Alexandroff’s theorem \cite{acg.Alex, acg..12}, for any finite set $X$, there is a bijection between topologies on $X$ and quasi-orders on $X$. 
Any topology $\mathcal{T}$ on $X$ defines a quasi-order denoted by $\leq_{\mathcal{T}}$ on $X$:
	\begin{equation*}
	x\leq_{\mathcal{T}}y\Longleftrightarrow \hbox{ any open subset containing $x$ also contains $y$}.
	\end{equation*}
	Conversely, any quasi-order $\leq$ on $X$ defines a topology $\mathcal{T}_{\leq}$ given by its upper ideals, i.e., subsets $Y\subset X$ such that ($y\in Y$ and $y\leq z$) $\implies z\in Y$. Both operations are inverse to each other:
	\begin{equation*}
	\leq_{\mathcal{T}_{\leq}}= \leq\hspace*{1cm} and \hspace*{1cm} \mathcal{T}_{\leq_{\mathcal{T}}}=\mathcal{T}.
	\end{equation*}
		Hence there is a natural bijection between topologies and quasi-orders on
	a finite set $X$.
	Any quasi-order (hence any topology $\mathcal{T}$ ) on $X$ gives rise to an equivalence relation:
	\begin{equation}
	x \sim_{\mathcal{T}}y\Longleftrightarrow \left( x\leq_{\mathcal{T}}y \hbox{ and } y\leq_{\mathcal{T}}x \right).
	\end{equation}
	A finite topological space $(X, \le)$ will be represented by the Hasse diagram of the quotient $X/\sim $, where $\sim$ is the equivalence relation defined above. Each vertex is drawn as a bubble in which  all elements of the same equivalence class are represented by points.\\
	 More on finite topological spaces can be found in \cite{Moh. twisted, Moh. Doubling, acg15}.\\
	 


Let $(Q, \le)$ be a topological space equipped with a continuous map $\mu : Q \times Q \longrightarrow Q$ , denoted by  $\mu(a, b) = a \lhd b$, such that for every $b\in Q$ the mapping $a\mapsto a \lhd b$ is a homeomorphism of $(Q,\le)$. The space $Q$ (together with the map $\mu$ ) is called a topological quandle \cite{Rubinsztein} if it satisfies for all $a, b, c \in Q$
\begin{itemize}
    \item (i) $(a \lhd b) \lhd c=(a \lhd c) \lhd (b \lhd c)$,
    \item (ii) $a \lhd a=a$.
\end{itemize}

Let $(Q, \lhd)$ and $(Q', \lhd')$ be two topological quandles.
A continuous map $\phi : Q \longrightarrow Q'$ is called a topological quandle homomorphism if $\phi(a\lhd b) = \phi(a) \lhd' \phi(b)$, for all $a, b \in Q$.\\

The paper is organized as follows. We recall in Section \ref{matrix of a finite quandle} the method of B. Ho and S. Nelson  \cite{B. Ho and S. Nelson} to describe finite quandles with up to 5 elements,
and we also recall in section \ref{orbit} how S. Nelson and C-Y. Wong in \cite{Nelson and Wong} prove that the decomposition of a finite quandle into orbits coincides with our notion of decomposition into Q-complemented subquandles.\\
 
 In section \ref{results} we prove that, if $Q=Q_1 \amalg Q_2 \amalg \cdot \cdot \cdot \amalg Q_n$ is a finite quandle, written in its orbit decomposition, and if $\mathcal{T}=(Q, \le)$ is a topological space such as $\mathcal{T}_{|Q_i}$ is the coarse topology on $Q_i$ for all $i\in [n]$, then $\mathcal{T}$ is $Q$-compatible. Then we apply this result to find the finite topological quandles with up to 4 elements.

\section{The matrix of a finite quandle}\label{matrix of a finite quandle}
Let $Q=\{x_1, x_2, . . . , x_n\}$ be a finite quandle with $n$ elements. We define the
matrix of $Q$, denoted $M_Q$, to be the matrix whose entry in row $i$ column $j$ is $x_i \lhd x_j$:\\
$$M_Q=	\begin{bmatrix}
x_1\lhd x_1 & x_1\lhd x_2 &...& x_1\lhd x_n\\
x_2\lhd x_1 & x_2\lhd x_2 &...& x_2\lhd x_n\\
.&.&...&.\\
.&.&...&.\\
.&.&...&.\\
x_n\lhd x_1 & x_n\lhd x_2 &...& x_n\lhd x_n
\end{bmatrix}
$$
\begin{examples}\cite{B. Ho and S. Nelson}
    Let $Q=\{a, b, c \}$, the Quandle matrices for quandles of order 3 are, up to permutations of $Q$:
   $$\begin{bmatrix}
a & a & a\\
b & b & b\\
c & c & c
\end{bmatrix}
, \hspace{1cm}
\begin{bmatrix}
a & c & b\\
c & b & a\\
b & a & c
\end{bmatrix}
, \hspace{1cm}
\begin{bmatrix}
a & a & a\\
c & b & b\\
b & c & c
\end{bmatrix}
$$ 
Let $Q=\{a, b, c, d \}$, the Quandle matrices for quandles of order 4 are, up to permutations of $Q$:
   $$\begin{bmatrix}
a & a & a & a\\
b & b & b & b\\
c & c & c & c\\
d & d & d & d
\end{bmatrix}
, \hspace{0.3cm}
\begin{bmatrix}
a & a & a & a\\
b & b & b & c\\
c & c & c & b\\
d & d & d & d
\end{bmatrix}
, \hspace{0.3cm}
\begin{bmatrix}
a & a & a & b\\
b & b & b & c\\
c & c & c & a\\
d & d & d & d
\end{bmatrix}
, \hspace{0.3cm}
\begin{bmatrix}
a & a & b & b\\
b & b & a & a\\
c & c & c & c\\
d & d & d & d
\end{bmatrix}
, $$
$$
\begin{bmatrix}
a & a & a & a\\
b & b & d & c\\
c & d & c & b\\
d & c & b & d
\end{bmatrix}
, \hspace{0.3cm}
\begin{bmatrix}
a & a & b & b\\
b & b & a & a\\
d & d & c & c\\
c & c & d & d
\end{bmatrix}
, \hspace{0.3cm}
\begin{bmatrix}
a & d & b & c\\
c & b & d & a\\
d & a & c & b\\
b & c & a & d
\end{bmatrix}
$$

\end{examples}	
\begin{definition}
Let $Q$ be a quandle. A subquandle $X \subset Q$ is a subset of $Q$ which is itself a quandle under $\lhd$. Let $Q$ be a quandle and $X \subset Q$ a subquandle. We say that $X$ is complemented in $Q$ or $Q$-complemented if $Q\backslash X$ is a subquandle of $Q$. 
\end{definition}
\begin{theorem}\cite{Nelson and Wong}
Let $Q$ be a finite quandle. Then $Q$ may be written as
$$Q = Q_1 \amalg Q_2 \amalg \cdot \cdot \cdot \amalg Q_n,$$
where every $Q_i$ is $Q$-complemented and no proper subquandle of any $Q_i$ is $Q$-complemented. This
decomposition is well-defined up to isomorphism; if $Q \approx Q'$, then in the decompositions
$Q = Q_1 \amalg Q_2 \amalg \cdot \cdot \cdot \amalg Q_n,$ and $Q' = Q'_1 \amalg Q'_2 \amalg \cdot \cdot \cdot \amalg Q'_m,$
we have $n = m$ and (after reordering if necessary), $Q_i=Q'_j$.
\end{theorem}
\section{Reminder on the orbit decomposition}\label{orbit}



\textbf{Notation.} Let $(Q, \lhd)$ be a finite quandle, for $x'\in Q$, we note 

\begin{equation*}
			   \begin{split}
			         R_{x'}:Q &\longrightarrow Q\\
		              x &\longmapsto x\lhd x',
			    \end{split}
			     \ \ \ \ \ \ \ \ \ \ \ \ \ \ \ and \ \ \ \ \ \ \ \ \ \ \ \ \ \ \begin{split}
			         L_{x'}:Q& \longrightarrow Q\\
		         x& \longmapsto x'\lhd x.
			    \end{split}
			\end{equation*}

	\begin{remark}
	$(Q, \mathcal{T})$ is a finite topological quandle if and only if, ($R_{x'}$ is an homeomorphism and $L_{x'}$ is a continuous map, for all $x'\in Q$) if and only if, for all $x, y, x', y' \in X$, if $x\le x'$ and $y\le y'$, we obtain $x\lhd y\le x'\lhd y'$.                
	\end{remark}
 \begin{lemma}\label{lemquandle}
Let $(Q, \lhd)$ be a finite quandle, the intersection of two $Q$-complemented subquandles is also $Q$-complemented.
\end{lemma}
\begin{proof}
Let $(Q, \lhd)$ be a finite quandle and let $Q_1$, $Q_2$ be two $Q$-complemented sub-quandles.\\
It is clear that the binary operation $\lhd : (Q_1\cap Q_2) \times (Q_1\cap Q_2) \longrightarrow Q_1\cap Q_2$ satisfies the two axioms $(i)$ and $(iii)$ of the definition of quandle. \\
 For $x, y\in Q_1\cap Q_2$, it exists $z\in Q$ such as $x=R_y(z)$, where $R_y:Q\longrightarrow Q$ defined by $R_y(z)=z\lhd y$. i.e., $z=R^{-1}_y(x)$. Since $x, y\in Q_1\cap Q_2$ and the map $R_y$ is a bijection on $Q_1$ (resp. on $Q_2$), so we get $z\in Q_1\cap Q_2$. Hence $\lhd : (Q_1\cap Q_2) \times (Q_1\cap Q_2) \longrightarrow Q_1\cap Q_2$ satisfies the axiom $(ii)$. So $Q_1\cap Q_2$ is a sub-quandle.\\
 On the other hand:
 $$Q=(Q_1\cap Q_2)\amalg (Q_1\cap \overline{Q_2})\amalg (\overline{Q_1}\cap Q_2)\amalg (\overline{Q_1}\cap \overline{Q_2}), \hbox{ where } \overline{Q_1}=Q\backslash Q_1 \hbox{ and } \overline{Q_2}=Q\backslash Q_2.$$
Let $a\in \overline{Q_1\cap Q_2}$, so we have three possible cases; $a\in Q_1\cap \overline{Q_2}$ or $a\in \overline{Q_1}\cap Q_2$ or $a\in \overline{Q_1}\cap \overline{Q_2}$.
 \begin{itemize}
     \item If $a\in \overline{Q_1}\cap \overline{Q_2}$, we obtain 
     \begin{itemize}
         \item $R_a: \overline{Q_1}\cap \overline{Q_2}\longmapsto \overline{Q_1}\cap \overline{Q_2}$ is a bijection.
         \item $R_a: \overline{Q_1}\longmapsto \overline{Q_1}$ is a bijection.
          \item $R_a: \overline{Q_2}\longmapsto \overline{Q_2}$ is a bijection.
           \item $R_a: Q\longmapsto Q$ is a bijection.
     \end{itemize}
     Then $R_a$ respects all four blocks.
  \item If $a\in \overline{Q_1}\cap Q_2$ or $a\in Q_1\cap \overline{Q_2}$; similarly.
 \end{itemize}
 Hence $R_a$ respects $\overline{Q_1\cap Q_2}$, so we then deduce that $Q_1\cap Q_2$ is a $Q$-complemented subquandles.\\
 Then the finite intersection of $Q$-complemented subquandles is also $Q$-complemented.
\end{proof}
\textbf{Notation:}
Let $(Q, \lhd)$ be a finite quandle. For $a\in Q$, we not
$$Q_a=\bigcap \limits_{\underset{a\in Q' }{Q'\hbox{\tiny{ is Q-complemented }}}}Q'
\hspace{1cm}\hbox{ and }\hspace{1cm} 
\Omega_a=\{ b\in Q, a\sim b  \},$$ 
where $\sim$ is the transitive closure of the relation $\tilde{\mathcal{R}}$ defined by:
$$x\tilde{\mathcal{R}}y\Longleftrightarrow \hbox{ it exists } z\in Q \hbox{ such as } (x=y\lhd z  \hbox{ or }  y=x\lhd z).$$
i.e., for all $a, b\in Q,$
$$a\sim b \hbox{ if and only if, it exists } c_1,..., c_n\in Q \hbox{ such as } a\tilde{\mathcal{R}}c_1...\tilde{\mathcal{R}}c_n\tilde{\mathcal{R}}b.$$
\begin{theorem}\cite{Nelson and Wong}\label{lemquandleorbit}
Let $(Q, \lhd)$ be a finite quandle then $\Omega_a$ and $Q_a$ defined above are equal for any $a\in Q$.
\end{theorem}
\begin{proof}
Let $(Q, \lhd)$ be a finite quandle and $a\in Q$, according to the Lemma \ref{lemquandle}, $Q_a$ is a $Q$-complemented subquandle.\\
 - It is clear that the binary operation $\lhd: \Omega_a \times \Omega_a \longrightarrow \Omega_a$ satisfies the two axioms $(i)$ and $(iii)$ of the definition of quandle.\\
 Let $x, y\in \Omega_a$, then there exists a unique $z \in Q$ such that $x=z\lhd y$, and hence $x\tilde{\mathcal{R}}z$, hence $z\in \Omega_x=\Omega_a$. Hence, the map $R_x: \Omega_a \longrightarrow \Omega_a$ defined by $R_x(y)=y\lhd x$ is a bijection. So $\Omega_a$ is a sub-quandle of $Q$.\\
Moreover the binary operation $\lhd: Q\backslash \Omega_a \times Q \backslash\Omega_a \longrightarrow Q\backslash\Omega_a$ satisfies the two axioms $(i)$ and $(iii)$ of the definition of quandle. And for all $x, y\in Q\backslash \Omega_a$ there exists $z\in Q$ such that $x=z\lhd y$, hence $x\tilde{\mathcal{R}}z$, then $z\in Q\backslash \Omega_a$ necessarily, because otherwise then $x\in \Omega_a$, which is absurd. Hence, the map $R_x :Q\backslash \Omega_a \longrightarrow Q\backslash\Omega_a$ defined by $R_x(y)=y\lhd x$ is a bijection. So $Q\backslash\Omega_a$ is a sub quandle of $Q$. then $\Omega_a$ is $Q$-complemented.\\
 - Since $Q_a$ is the smallest complemented sub-quandle containing $a$, we obtain that $Q_a\subseteq \Omega_a$.\\
 - It remains to show that $\Omega_a \subseteq Q_a$, let $B$ be a sub-quandle $Q$-complemented containing $a$. For $x\in B$, then $R_x$ respects $B$. Moreover for $x\in \overline{B}$, $R_x$ respects $\overline{B}$.
  So for all $x\in Q$, $R_x$ and $R^{-1}_x$ respect $B$. And since $\Omega=\{P_1...P_ka, \hbox{ with } P_j \hbox{ equal to } R_{x_j} \hbox{ or } R^{-1}_{x_j}, x_j\in Q \}$, then $\Omega_a\subset B$.\\
Hence $\Omega_a\subset Q_a$. Consequently $\Omega_a=Q_a$.
\end{proof}

\section{Results}\label{results}
In this section we prove that, if $Q=Q_1 \amalg Q_2 \amalg \cdot \cdot \cdot \amalg Q_n$ is a finite quandle and let $\mathcal{T}=(Q, \le)$ is a topological space such as for all $i\in [n]$, $\mathcal{T}_{|Q_i}$ is the coarse topology on $Q_i$, then $\mathcal{T}$ is $Q$-compatible. From this  result I find the topological quandles of 3 and 4 elements.
\subsection{The topologies of orbits of finite quandle}

	\begin{proposition}\label{proposition quandle}
	Let $Q$ be a finite quandle, then the discrete topology and the coarse topology are $Q$-compatible.
	\end{proposition}
	\begin{proof}
Let $Q=(X, \lhd)$ be a finite quandle. If $\mathcal{T}$ is the discrete topology, then for all $x, x', y, y' \in X$, if $x\le_{\mathcal{T}} x'$ and $y\le_{\mathcal{T}} y'$, then $x=x'$ and $y=y'$, so $x\lhd y\le_{\mathcal{T}}x'\lhd y'$, hence $\mathcal{T}$ is $Q$-compatible.\\
If $\mathcal{T}$ is the coarse topology, then for all $x, y \in X$, $x\sim_{\mathcal{T}} y$, so for all $x, x', y, y'\in X$, $x\le_{\mathcal{T}} x'$ and $y\le_{\mathcal{T}} y'$ and $x\lhd y\le_{\mathcal{T}}x'\lhd y'$. Hence $\mathcal{T}$ is a $Q$-compatible.\\
		\end{proof}
	\textbf{Notation.} Let $Q=Q_1 \amalg Q_2 \amalg \cdot \cdot \cdot \amalg Q_n$ be a finite quandle written in its orbit decomposition (see Theorem \ref{lemquandleorbit}). We denote by $\mathcal{T}_Q=\mathcal{T}_{Q_1}\cdot \cdot \cdot\mathcal{T}_{Q_n}$ the usual product topology of $\mathcal{T}_{Q_i}, i\in [n]$, where $\mathcal{T}_{Q_i}$ is the coarse topology on $Q_i$.
\begin{example}
Let $\lhd:X\times X\longrightarrow X$ the operation of the quandle $Q$ defined by
$M_Q=\begin{bmatrix}
a & a & a\\
c & b & b\\
b & c & c
\end{bmatrix}$. Its orbit decomposition is $Q=Q_1\amalg Q_2$, where $Q_1=\begin{bmatrix}
a
\end{bmatrix}$ and $Q_2=\begin{bmatrix}
 b & b\\
 c & c
\end{bmatrix}$.
In this case 
$$\mathcal{T}_Q=\fcolorbox{white}{white}{
\scalebox{0.7}{
  \begin{picture}(73,47) (183,-169)
    \SetWidth{1.0}
    \SetColor{Black}
    \Vertex(194,-154){2}
    \Vertex(204,-154){2}
    \Text(191,-150)[lb]{\Large{\Black{$b$}}}
    \Text(202,-150)[lb]{\Large{\Black{$c$}}}
    \Arc(200,-152)(15.811,215,575)
    \Vertex(174,-154){2}
    \Text(174,-150)[lb]{\Large{\Black{$a$}}}
  \end{picture}
}}$$
\end{example}
	\begin{theorem}\label{theorem quandle}
	  Let $Q=Q_1 \amalg Q_2 \amalg \cdot \cdot \cdot \amalg Q_n$ be a finite quandle written in its orbit decomposition, and let $\mathcal{T}$ be a topology on $Q$.
	  If for all $i\in [n]$, $\mathcal{T}_{|Q_i}$ is the coarse topology on $Q_i$, then $\mathcal{T}$ is $Q$-compatible.
	\end{theorem}
	\begin{proof}
Let $\mathcal{T}$ be a topology on $Q$, such that $\mathcal{T }_{|Q_i}$ is the coarse topology, for all $i\in [ n ]$. For $x\in Q_i$, we note $Q_x=Q_i$.\\
Let $z, z'\in Q$ such that $z\le z'$, then for all $x\in Q$, $L_x(z)=x\lhd z\in Q_x$ and $L_x(z' )=x\lhd z'\in Q_x$. But $\mathcal{T}_{Q_i}$ is the coarse topology, then for all $a, b\in Q_i, a\sim b$, hence $x\lhd z \sim x\lhd z'$. Hence the continuity of $L_x$ for all $x\in Q$ is proven. Moreover $z\le z'$ implies that, for all $a\in Q_z, b\in Q_{z'}$, $a\le b$. In particular, $R_x(z)=z\lhd x\in Q_z$ and $R_x(z')=z'\lhd x\in Q_{z'}$, hence we get $R_x(z)\le R_x(z' )$. Hence $R_x$ is continuous for all $x\in Q$. As $\mathcal{T}$ is finite,
we therefore conclude that $\mathcal{T}$ is $Q$-compatible.
	\end{proof}
I use this theorem to find the topological quandles of 3 and 4 elements below.
\subsection{List of the topological quandles with three elements}
In the three examples above $X=\{a,b,c\}$.

- Let $\lhd:X\times X\longrightarrow X$ the operation of the trivial quandle $Q$ defined by
$M_Q=\begin{bmatrix}
a & a & a\\
b & b & b\\
c & c & c
\end{bmatrix}$. All topologies on $X$ are compatible with this quandle structure.\\
Indeed: let $\mathcal{T}$ be a topology on $X$, for all $x,y\in X$, $x\lhd y=y$, then for all $x',y'\in X$ such that $x\le x'$ and $y\le y'$, we obtain $x\lhd y\le x'\lhd y'$.

- Let $\lhd:X\times X\longrightarrow X$ the quandle structure defined by
$M_Q=\begin{bmatrix}
a & c & b\\
c & b & a\\
b & a & c
\end{bmatrix}$, 
let $\mathcal{T}=(X, \le)$ be a $Q$-compatible topology, if there exists $x\ne y\in \{a, b, c\}$ such that $x\le y$, then $\mathcal{T}$ is the coarse topology. In fact, suppose $a\le b$ we get,\\
$R_a(a)=a\le R_a(b)=c$ and $R_b(a)=c\le R_b(b)=b$ and $R_c(a)=b\le R_c(b)=a$, we therefore obtain, $a\le b$ implies that $a\le c\le b\le a$, hence $\mathcal{T}$ is the coarse topology. Similarly if we change $a, b$ by $x, y\in \{a, b, c\}$, we find that $\mathcal{T}$ is the coarse topology.\\
From Proposition \ref{proposition quandle},we conclude in this case that the topologies on $X$ compatible with the structure $\lhd$ are: the discrete topology and the coarse topology.\\

- Let $\lhd:X\times X\longrightarrow X$ the quandle structure defined by
$M_Q=\begin{bmatrix}
a & a & a\\
c & b & b\\
b & c & c
\end{bmatrix}$,\\
then according to Theorem \ref{theorem quandle}, the four topologies below endowed with $\lhd$ are compatible with the structure of quandle.
$$
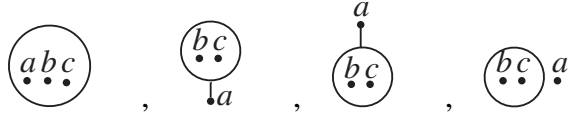
\fcolorbox{white}{white}{
\scalebox{0.7}{
  \begin{picture}(65,49) (172,-166)
    \SetWidth{1.0}
    \SetColor{Black}
    \Arc(195,-143)(21.633,56,416)
    \Vertex(183,-151){2}
    \Vertex(194,-151){2}
    \Vertex(204,-152){2}
    \Text(180,-147)[lb]{\Large{\Black{$a$}}}
    \Text(191,-147)[lb]{\Large{\Black{$b$}}}
    \Text(202,-147)[lb]{\Large{\Black{$c$}}}
  \end{picture}
}}
,
\fcolorbox{white}{white}{
\scalebox{0.7}{
  \begin{picture}(54,61) (183,-169)
    \SetWidth{1.0}
    \SetColor{Black}
    \Vertex(194,-142){2}
    \Vertex(204,-142){2}
    \Text(191,-138)[lb]{\Large{\Black{$b$}}}
    \Text(202,-138)[lb]{\Large{\Black{$c$}}}
    \Arc(200,-138)(15.811,215,575)
    \Line(200,-155)(200,-163)
    \Vertex(200,-165){2}
    \Text(204,-168)[lb]{\Large{\Black{$a$}}}
  \end{picture}
}},
\fcolorbox{white}{white}{
\scalebox{0.7}{
  \begin{picture}(54,67) (183,-149)
    \SetWidth{1.0}
    \SetColor{Black}
    \Vertex(194,-136){2}
    \Vertex(204,-136){2}
    \Text(191,-133)[lb]{\Large{\Black{$b$}}}
    \Text(202,-133)[lb]{\Large{\Black{$c$}}}
    \Arc(200,-132)(15.811,215,575)
    \Line(199,-116)(199,-105)
    \Vertex(199,-104){2}
    \Text(196,-100)[lb]{\Large{\Black{$a$}}}
  \end{picture}
}},
\fcolorbox{white}{white}{
\scalebox{0.7}{
  \begin{picture}(73,47) (183,-169)
    \SetWidth{1.0}
    \SetColor{Black}
    \Vertex(194,-154){2}
    \Vertex(204,-154){2}
    \Text(191,-150)[lb]{\Large{\Black{$b$}}}
    \Text(202,-150)[lb]{\Large{\Black{$c$}}}
    \Arc(200,-152)(15.811,215,575)
    \Vertex(223,-154){2}
    \Text(221,-150)[lb]{\Large{\Black{$a$}}}
  \end{picture}
}}
$$
Let $\mathcal{T}=(X, \le)$ be a $Q$-compatible topology, then ($b\sim c$ or $b$ and $c$ are incomparable). Indeed: if $b\le c$, then $R_a(b)=c\le R_a(c)=b$, similarly if $c\le b$, then $R_a(c)=b\le R_a (b)=c$. So the result.\\
Let $\mathcal{T}=(X, \le)$ be a $Q$-compatible topology such that $c$ and $b$ are incomparable then $\mathcal{T}$ is the discrete topology. In fact ; if $a\le b$ then $L_b(a)=c\le b=L_b(b)$, which is absurd, moreover, if $a\le c$ then $L_c(a)=b\le c=L_c(c)$ which is absurd (same if $b\le a$ or $c\le a$). Hence $\mathcal{T}$ is the discrete topology.\\
We conclude that the discrete topology
$\fcolorbox{white}{white}{
\scalebox{0.7}{
  \begin{picture}(58,39) (179,-166)
    \SetWidth{1.0}
    \SetColor{Black}
    \Vertex(183,-161){2}
    \Vertex(194,-161){2}
    \Vertex(204,-162){2}
    \Text(180,-156)[lb]{\Large{\Black{$a$}}}
    \Text(191,-156)[lb]{\Large{\Black{$b$}}}
    \Text(202,-156)[lb]{\Large{\Black{$c$}}}
  \end{picture}
}}$ and the above four topologies are the only $Q$-compatible topologies.
	
\subsection{List of the topological quandles with four elements}

In the seven examples below $X=\{a, b, c, d\}$.\\

- Let $\lhd:X\times X\longrightarrow X$ the quandle structure defined by
$M_Q=\begin{bmatrix}
a & d & b & c\\
c & b & d & a\\
d & a & c & b\\
b & c & a & d
\end{bmatrix}$, the only topologies on $X$ compatible with the quandle structure are the discrete topology and the coarse topology.\\
Indeed: let $(Q, \mathcal{T})$ be a topological quandle different from the discrete topology, then there exists $x\ne y \in \{a, b, c, d\}$, such that $x\le y$.\\
If $a\le b$, then $R_a(a)=a\le c=R_a(b)$, $R_b(a)=d\le b=R_b(b)$, $R_c(a)=b\le d=R_c(b)$, $R_d(a)=c\le a=R_d(b)$, $L_a(a)=a\le d=L_a(b)$, $L_b(a)=c\le b=L_b(b)$, $L_c(a)=d\le a=L_c(b)$ and $L_d(a)=b\le c=L_d(b)$. Then, $a\sim b\sim c\sim d$, i.e., $\mathcal{T}$ is a coarse topology.\\
Same if $a\le c,$ or $ a\le d, $ or $ b\le a, $ or $ b\le c,$ or $  b\le d,$ or $  c\le a,$ or $  c\le b, $ or $ c\le d,$ or $  d\le a,$ or $  d\le b,$ or $ d\le c$, we prove that $\mathcal{T}$ is a coarse topology.\\

- Let $\lhd:X\times X\longrightarrow X$ the quandle structure defined by
$M_Q=\begin{bmatrix}
a & a & b & b\\
b & b & a & a\\
d & d & c & c\\
c & c & d & d
\end{bmatrix}$,\\
If $(Q, \mathcal{T})$ is a topological quandle, then $ (a\sim b \hbox{ and } c\sim d) $
$\hbox{ or } (a\sim b \hbox{ and } c, d \hbox{ are incomparable})$ $ \hbox{ or } (a, b \hbox{ are incomparable and } c\sim d) \hbox{ or } (a, b \hbox{ are incomparable and c, d \hbox{ are incomparable}})$.\\
Indeed, if $a\le b$, then $R_c(a)=b\le a=R_c(b)$. So $a\sim b$.\\
Similarly, if $b\le a$, then $a\sim b$. If $c\le d$, then $c\sim d$. If $d\le c$, then $c\sim d$. \\
By Theorem \ref{theorem quandle}, the three topologies below are Q-compatible.\\
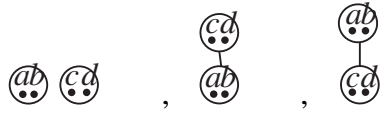
$$
\fcolorbox{white}{white}{
\scalebox{0.7}{
   \begin{picture}(75,36) (133,-200)
    \SetWidth{1.0}
    \SetColor{Black}
    \Arc(144,-189)(10,180,540)
    \Arc(171,-189)(10,180,540)
    \Vertex(141,-194){2}
    \Vertex(147,-194){2}
    \Vertex(167,-193){2}
    \Vertex(175,-193){2}
    \Text(137,-190)[lb]{\Large{\Black{$a$}}}
    \Text(145,-190)[lb]{\Large{\Black{$b$}}}
    \Text(164,-190)[lb]{\Large{\Black{$c$}}}
    \Text(173,-190)[lb]{\Large{\Black{$d$}}}
  \end{picture}
}},
\fcolorbox{white}{white}{
\scalebox{0.7}{
  \begin{picture}(47,62) (133,-174)
    \SetWidth{1.0}
    \SetColor{Black}
    \Arc(144,-163)(10,180,540)
    \Vertex(141,-168){2}
    \Vertex(147,-168){2}
    \Text(137,-165)[lb]{\Large{\Black{$a$}}}
    \Text(145,-165)[lb]{\Large{\Black{$b$}}}
    \Arc(144,-135)(10.05,174,534)
    \Line(145,-153)(144,-145)
    \Vertex(140,-139){2}
    \Vertex(147,-139){2}
    \Text(137,-136)[lb]{\Large{\Black{$c$}}}
    \Text(145,-136)[lb]{\Large{\Black{$d$}}}
  \end{picture}
}},
\fcolorbox{white}{white}{
\scalebox{0.7}{
  \begin{picture}(50,68) (132,-140)
    \SetWidth{1.0}
    \SetColor{Black}
    \Arc(144,-129)(10.05,174,534)
    \Vertex(140,-133){2}
    \Vertex(147,-133){2}
    \Text(137,-131)[lb]{\Large{\Black{$c$}}}
    \Text(145,-131)[lb]{\Large{\Black{$d$}}}
    \Arc(143,-96)(10.198,169,529)
    \Line(144,-119)(144,-106)
    \Vertex(140,-100){2}
    \Vertex(147,-100){2}
    \Text(137,-97)[lb]{\Large{\Black{$a$}}}
    \Text(145,-97)[lb]{\Large{\Black{$b$}}}
  \end{picture}
}}
$$
The disjoint union of the discrete topology on $\{a, b\}$ and the coarse topology on $\{c, d\}$ is Q-compatible and vice versa.\\
Let $\mathcal{T}=(X,\leq)$ be a topological space differs from the coarse topology, and suppose there exists $x\in \{a,b\}$ (resp. $x\in \{c,d\}$) and $y\in \{c,d\}$ (resp. $y\in \{a,b\}$) such that $x\leq y$. Then, $\mathcal{T}$ is not Q-compatible.
Indeed, if $\mathcal{T}$ is a Q-compatible wich $a\le c$ then $L_a(a)=a\le b=L_a(c)$, which is absurd.\\
Conclusion: there are seven topologies Q-compatible (the three topologies above and the 4 below).\\

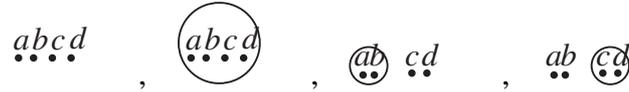
$$\fcolorbox{white}{white}{
\scalebox{0.7}{
  \begin{picture}(65,49) (172,-166)
    \SetWidth{1.0}
    \SetColor{Black}
    \Vertex(180,-151){2}
    \Vertex(188,-151){2}
    \Vertex(198,-151){2}
    \Vertex(208,-151){2}
    \Text(178,-147)[lb]{\Large{\Black{$a$}}}
    \Text(188,-147)[lb]{\Large{\Black{$b$}}}
    \Text(198,-147)[lb]{\Large{\Black{$c$}}}
     \Text(208,-147)[lb]{\Large{\Black{$d$}}}
  \end{picture}
}}, \fcolorbox{white}{white}{
\scalebox{0.7}{
  \begin{picture}(65,49) (172,-166)
    \SetWidth{1.0}
    \SetColor{Black}
    \Arc(195,-143)(21.633,56,416)
    \Vertex(180,-151){2}
    \Vertex(188,-151){2}
    \Vertex(198,-151){2}
    \Vertex(208,-151){2}
    \Text(178,-147)[lb]{\Large{\Black{$a$}}}
    \Text(188,-147)[lb]{\Large{\Black{$b$}}}
    \Text(198,-147)[lb]{\Large{\Black{$c$}}}
     \Text(208,-147)[lb]{\Large{\Black{$d$}}}
  \end{picture}
}}, \fcolorbox{white}{white}{
\scalebox{0.7}{
   \begin{picture}(75,36) (133,-200)
    \SetWidth{1.0}
    \SetColor{Black}
    \Arc(144,-189)(10,180,540)
    \Vertex(141,-194){2}
    \Vertex(147,-194){2}
    \Vertex(167,-193){2}
    \Vertex(175,-193){2}
    \Text(137,-190)[lb]{\Large{\Black{$a$}}}
    \Text(145,-190)[lb]{\Large{\Black{$b$}}}
    \Text(164,-190)[lb]{\Large{\Black{$c$}}}
    \Text(173,-190)[lb]{\Large{\Black{$d$}}}
  \end{picture}
}}, \fcolorbox{white}{white}{
\scalebox{0.7}{
   \begin{picture}(75,36) (133,-200)
    \SetWidth{1.0}
    \SetColor{Black}
    \Arc(171,-189)(10,180,540)
    \Vertex(141,-194){2}
    \Vertex(147,-194){2}
    \Vertex(167,-193){2}
    \Vertex(175,-193){2}
    \Text(137,-190)[lb]{\Large{\Black{$a$}}}
    \Text(145,-190)[lb]{\Large{\Black{$b$}}}
    \Text(164,-190)[lb]{\Large{\Black{$c$}}}
    \Text(173,-190)[lb]{\Large{\Black{$d$}}}
  \end{picture}
}}$$
- Let $\lhd:X\times X\longrightarrow X$ the quandle structure defined by
$M_Q=\begin{bmatrix}
a & a & a & a\\
b & b & d & c\\
c & d & c & b\\
d & c & b & d
\end{bmatrix}$, then $Q=Q_1\amalg Q_2$, where $Q_1=\begin{bmatrix}
a
\end{bmatrix}$ and $Q_2=\begin{bmatrix}
 b & d &c\\
 d & c &b\\
 c & b &d
\end{bmatrix}$\\
If $(Q, \mathcal{T})$ is a topological quandle, then ($b\sim c\sim d$ or $b, c, d$ are incomparable).\\
If $b\le c$, then $R_a(b)=b\le c=R_a(c)$, $R_b(b)=b\le d=R_b(c)$, $R_c(b)=d\le c=R_c(c)$, $R_d(b)=c\le b=R_d(c)$, $L_a(b)=a\le a=L_a(c)$, $L_b(b)=b\le d=L_b(c)$, $L_c(b)=d\le c=L_c(c)$ and $L_d(b)=c\le b=L_d(c)$.
So $b\sim c$, $b\le d$ and $d\le c$, implies that $b\sim c$ and $b\sim d$, then $b\sim c\sim d$.\\
By Theorem \ref{theorem quandle}, the three topologies below are Q-compatible.\\
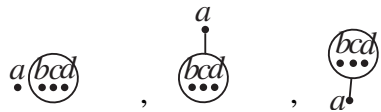
$$
\fcolorbox{white}{white}{
\scalebox{0.7}{
  \begin{picture}(61,35) (106,-198)
    \SetWidth{1.0}
    \SetColor{Black}
    \Arc(128,-184)(13.038,122,482)
    \Vertex(122,-189){2}
    \Vertex(128,-189){2}
    \Vertex(134,-189){2}
    \Vertex(109,-189){2}
    \Text(105,-185)[lb]{\Large{\Black{$a$}}}
    \Text(118,-185)[lb]{\Large{\Black{$b$}}}
    \Text(126,-185)[lb]{\Large{\Black{$c$}}}
    \Text(132,-185)[lb]{\Large{\Black{$d$}}}
  \end{picture}
}},
\fcolorbox{white}{white}{
\scalebox{0.7}{
  \begin{picture}(53,70) (114,-163)
    \SetWidth{1.0}
    \SetColor{Black}
    \Arc(128,-149)(13.038,122,482)
    \Vertex(122,-154){2}
    \Vertex(128,-154){2}
    \Vertex(134,-154){2}
    \Text(118,-150)[lb]{\Large{\Black{$b$}}}
    \Text(125,-150)[lb]{\Large{\Black{$c$}}}
    \Text(131,-150)[lb]{\Large{\Black{$d$}}}
    \Vertex(129,-123){2}
    \Line(129,-136)(129,-125)
    \Text(124,-119)[lb]{\Large{\Black{$a$}}}
  \end{picture}
}},
\fcolorbox{white}{white}{
\scalebox{0.7}{
  \begin{picture}(53,48) (114,-198)
    \SetWidth{1.0}
    \SetColor{Black}
    \Arc(128,-171)(13.038,122,482)
    \Vertex(121,-176){2}
    \Vertex(128,-176){2}
    \Vertex(134,-176){2}
    \Text(119,-171)[lb]{\Large{\Black{$b$}}}
    \Text(126,-171)[lb]{\Large{\Black{$c$}}}
    \Text(132,-172)[lb]{\Large{\Black{$d$}}}
    \Vertex(126,-196){2}
    \Line(126,-195)(127,-184)
    \Text(116,-202)[lb]{\Large{\Black{$a$}}}
  \end{picture}
}}
$$
Let $(Q, \mathcal{T})$ be a topological quandle such that, there exists $x\in \{b, c, d\}$ such that $a\le x$ or $x\le a$ then $\mathcal{T}$ is the coarse topology. Indeed: if $a\le b$, then   
$R_a(a)=a\le a=R_a(b)$, $R_b(a)=a\le b=R_b(b)$, $R_c(a)=c\le d=R_c(b)$, $R_d(a)=a\le c=R_d(b)$, $L_a(a)=a\le a=L_a(b)$, $L_b(a)=b\le b=L_b(b)$, $L_c(a)=c\le d=L_c(b)$ and $L_d(a)=d\le c=L_d(b)$. So $a\le d\le a$ and $c\le a\le c$, then $c\sim d$, then $a\sim c$ and $b\sim c\sim d$, so $a\sim b\sim c\sim d$.\\
Conclusion: there are five Q-compatible topologies: the coarse topology, the discrete topology and the three topologies described above. \\

- Let $\lhd:X\times X\longrightarrow X$ the quandle structure defined by
$M_Q=\begin{bmatrix}
a & a & b & b\\
b & b & a & a\\
c & c & c & c\\
d & d & d & d
\end{bmatrix}$, then $Q=Q_1\amalg Q_2\amalg Q_3$, where $Q_1=\begin{bmatrix}
a &a\\
b &b
\end{bmatrix}$, $Q_2=\begin{bmatrix}
 c
\end{bmatrix}$ and $Q_3=\begin{bmatrix}
 d
\end{bmatrix}$\\
If $(Q, \mathcal{T})$ is a topological quandle, then ($a\sim b$ or $a, b$ are incomparable).\\
Indeed, if $a\le b$, then $R_d(a)=b\le a=R_d(b)$, so $a\sim b$. Same thing if $b\le a$ then $a\sim b$.\\
By Theorem \ref{theorem quandle}, any topology that is coarse on the bags $\{a, b\}$, $\{c\}$, $\{d\}$ is Q-compatible.\\
If $a, b$ are incomparable: for all $x\in \{a, b\}$,
\begin{itemize}
    \item if $x\le d$, then $L_a(x)=a\le b=L_a(d)$, which is absurd, 
    \item if $d\le x$, then $L_a(d)=b\le a=L_a(x)$, which is absurd, 
    \item if $x\le c$, then $L_a(x)=a\le b=L_a(c)$, which is absurd, 
    \item if $c\le x$, then $L_a(c)=b\le a=L_a(x)$, which is absurd. 
\end{itemize}
Therefore, if $(Q, \mathcal{T})$ is a topological quandle with $a, b$ are incomparable, it implies that\\  
$\mathcal{T}=\fcolorbox{white}{white}{
\scalebox{0.7}{
  \begin{picture}(48,37) (156,-206)
    \SetWidth{1.0}
    \SetColor{Black}
    \Vertex(160,-202){2}
    \Vertex(160,-192){2}
    \Vertex(166,-202){2}
    \Vertex(174,-202){2}
    \Text(150,-204)[lb]{\Large{\Black{$c$}}}
    \Text(151,-190)[lb]{\Large{\Black{$d$}}}
    \Text(167,-200)[lb]{\Large{\Black{$a$}}}
    \Text(177,-200)[lb]{\Large{\Black{$b$}}}
    \Line(161,-202)(160,-193)
  \end{picture}
}}$, or \hspace{0.1cm} $\mathcal{T}=\fcolorbox{white}{white}{
\scalebox{0.7}{
  \begin{picture}(48,37) (156,-206)
    \SetWidth{1.0}
    \SetColor{Black}
    \Vertex(160,-202){2}
    \Vertex(160,-192){2}
    \Vertex(166,-202){2}
    \Vertex(174,-202){2}
    \Text(150,-204)[lb]{\Large{\Black{$d$}}}
    \Text(153,-190)[lb]{\Large{\Black{$c$}}}
    \Text(167,-200)[lb]{\Large{\Black{$a$}}}
    \Text(178,-200)[lb]{\Large{\Black{$b$}}}
    \Line(161,-202)(160,-193)
  \end{picture}
}}$, or\hspace{0.1cm} $\mathcal{T}=\fcolorbox{white}{white}{
\scalebox{0.7}{
  \begin{picture}(58,39) (179,-166)
    \SetWidth{1.0}
    \SetColor{Black}
    \Vertex(183,-161){2}
    \Vertex(194,-161){2}
    \Vertex(204,-161){2}
    \Vertex(214,-161){2}
    \Text(180,-156)[lb]{\Large{\Black{$a$}}}
    \Text(191,-156)[lb]{\Large{\Black{$b$}}}
    \Text(202,-156)[lb]{\Large{\Black{$c$}}}
    \Text(212,-156)[lb]{\Large{\Black{$d$}}}
  \end{picture}
}}
$, or \hspace{0.1cm} $\mathcal{T}=\hspace{0.3cm}\fcolorbox{white}{white}{
\scalebox{0.7}{
  \begin{picture}(73,47) (183,-169)
    \SetWidth{1.0}
    \SetColor{Black}
    \Vertex(194,-154){2}
    \Vertex(204,-154){2}
    \Text(191,-150)[lb]{\Large{\Black{$c$}}}
    \Text(202,-150)[lb]{\Large{\Black{$d$}}}
    \Arc(200,-152)(15.811,215,575)
    \Vertex(164,-154){2}
    \Vertex(174,-154){2}
    \Text(164,-150)[lb]{\Large{\Black{$a$}}}
    \Text(174,-150)[lb]{\Large{\Black{$b$}}}
  \end{picture}
}}$\\

It is clear that the above topologies are Q-compatible.\\
Conclusion: The Q-compatible topologies are the four topologies above and all topologies on $X$ such that $a$ and $b$ are equivalent.

- Let $\lhd:X\times X\longrightarrow X$ the quandle structure defined by
$M_Q=\begin{bmatrix}
a & a & a & b\\
b & b & b & c\\
c & c & c & a\\
d & d & d & d
\end{bmatrix}$, then $Q=Q_1\amalg Q_2$, where $Q_1=\begin{bmatrix}
 a & a &a\\
 b & b &b\\
 c & c &c
\end{bmatrix}$ and $Q_2=\begin{bmatrix}
d
\end{bmatrix}$.\\
If $(Q, \mathcal{T})$ is a topological quandle, then ($a\sim b\sim c$ or $a, b, c$ are incomparable). Indeed:\\
if $a\le b$, then $R_d(a)=b\le c=R_d(b)$, then $R_d(b)=c\le a=R_d(c)$. So $a\le b\le c\le a$, then $a\sim b\sim c$. Similarly for $x, y\in \{ a, b, c\}$, if $x\le y$ then $a\sim b\sim c$. Then the result.\\
$(Q, \mathcal{T})$ be a topological quandle with $a, b, c$ are incomparable, implies that $\mathcal{T}$ is the discrete topology. Indeed, if there exists $x\in \{a, b, c\}$ such that, ($x\le d$ or $d\le x$), then $\big( L_a(x)=a\le b=L_a(d)$ or $L_a(d)=b\le a=L_a(x)\big)$, which is absurd.\\ 
Conclusion: The Q-compatible topologies are the four topologies below.

$$
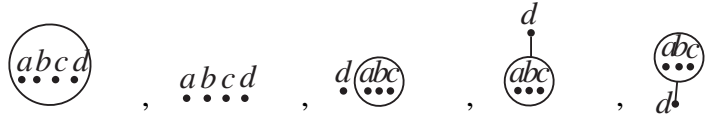
\fcolorbox{white}{white}{
\scalebox{0.7}{
  \begin{picture}(65,49) (172,-166)
    \SetWidth{1.0}
    \SetColor{Black}
    \Arc(195,-143)(21.633,56,416)
    \Vertex(180,-151){2}
    \Vertex(188,-151){2}
    \Vertex(198,-151){2}
    \Vertex(208,-151){2}
    \Text(178,-147)[lb]{\Large{\Black{$a$}}}
    \Text(188,-147)[lb]{\Large{\Black{$b$}}}
    \Text(198,-147)[lb]{\Large{\Black{$c$}}}
     \Text(208,-147)[lb]{\Large{\Black{$d$}}}
  \end{picture}
}},
\fcolorbox{white}{white}{
\scalebox{0.7}{
  \begin{picture}(58,39) (179,-166)
    \SetWidth{1.0}
    \SetColor{Black}
    \Vertex(183,-161){2}
    \Vertex(194,-161){2}
    \Vertex(204,-161){2}
    \Vertex(214,-161){2}
    \Text(180,-156)[lb]{\Large{\Black{$a$}}}
    \Text(191,-156)[lb]{\Large{\Black{$b$}}}
    \Text(202,-156)[lb]{\Large{\Black{$c$}}}
    \Text(212,-156)[lb]{\Large{\Black{$d$}}}
  \end{picture}
}},
\fcolorbox{white}{white}{
\scalebox{0.7}{
  \begin{picture}(61,35) (106,-198)
    \SetWidth{1.0}
    \SetColor{Black}
    \Arc(128,-184)(13.038,122,482)
    \Vertex(122,-189){2}
    \Vertex(128,-189){2}
    \Vertex(134,-189){2}
    \Vertex(109,-189){2}
    \Text(105,-185)[lb]{\Large{\Black{$d$}}}
    \Text(118,-185)[lb]{\Large{\Black{$a$}}}
    \Text(125,-185)[lb]{\Large{\Black{$b$}}}
    \Text(132,-185)[lb]{\Large{\Black{$c$}}}
  \end{picture}
}},
\fcolorbox{white}{white}{
\scalebox{0.7}{
  \begin{picture}(53,70) (114,-163)
    \SetWidth{1.0}
    \SetColor{Black}
    \Arc(128,-149)(13.038,122,482)
    \Vertex(122,-154){2}
    \Vertex(128,-154){2}
    \Vertex(134,-154){2}
    \Text(118,-150)[lb]{\Large{\Black{$a$}}}
    \Text(125,-150)[lb]{\Large{\Black{$b$}}}
    \Text(132,-150)[lb]{\Large{\Black{$c$}}}
    \Vertex(129,-123){2}
    \Line(129,-136)(129,-125)
    \Text(124,-119)[lb]{\Large{\Black{$d$}}}
  \end{picture}
}},
\fcolorbox{white}{white}{
\scalebox{0.7}{
  \begin{picture}(53,48) (114,-198)
    \SetWidth{1.0}
    \SetColor{Black}
    \Arc(128,-171)(13.038,122,482)
    \Vertex(121,-176){2}
    \Vertex(128,-176){2}
    \Vertex(134,-176){2}
    \Text(119,-171)[lb]{\Large{\Black{$a$}}}
    \Text(125,-171)[lb]{\Large{\Black{$b$}}}
    \Text(132,-172)[lb]{\Large{\Black{$c$}}}
    \Vertex(126,-196){2}
    \Line(126,-195)(127,-184)
    \Text(116,-202)[lb]{\Large{\Black{$d$}}}
  \end{picture}
}}
$$

- Let $\lhd:X\times X\longrightarrow X$ the quandle structure defined by
$M_Q=\begin{bmatrix}
a & a & a & a\\
b & b & b & c\\
c & c & c & b\\
d & d & d & d
\end{bmatrix}$,
then $Q=Q_1\amalg Q_2\amalg Q_3$, where $Q_1=\begin{bmatrix}
 a
\end{bmatrix}$, $Q_2=\begin{bmatrix}
b &b\\
c &c
\end{bmatrix}$ and $Q_3=\begin{bmatrix}
 d
\end{bmatrix}$.
\\
If$(Q, \mathcal{T})$ is a topological quandle, then ($b\sim c$ or $b, c$ are incomparable). Indeed:\\
if $b\le c$, then $R_d(b)=c\le b=R_d(c)$, then $b\sim c$.\\
By Theorem \ref{theorem quandle}, the topology of the bags $\{a\}$ $\{b, c\}$, $\{d\}$ is a Q-compatible.\\
Let $(Q, \mathcal{T})$ be a topological quandle, then: $b, c$ are incomparable, implies that for all $x\in \{a, b, c\}$, $x$ and $d$ are incomparable. By absurd: if ($x\le d$ or $d\le x$) then $\big(L_b(x)=b\le c=L_b(d)$ or $L_b(d)=c\le d=L_b(x)\big)$, which is absurd. Moreover if ($a\le b$ or $a\le c$) then $\big(R_d(a)=a\le c=R_d(b)$ or $R_d(a)=a\le b=R_d(c)\big)$. We deduce therefore that: $(Q, \mathcal{T})$ be a topological quandle which $b, c$ are incomparable, implies that\\
$$\mathcal{T}=\fcolorbox{white}{white}{
\scalebox{0.7}{
  \begin{picture}(56,37) (150,-205)
    \SetWidth{1.0}
    \SetColor{Black}
    \Vertex(172,-203){2}
    \Vertex(160,-203){2}
    \Vertex(153,-191){2}
    \Vertex(167,-191){2}
    \Line(160,-203)(154,-192)
    \Line(160,-204)(167,-191)
    \Text(150,-207)[lb]{\Large{\Black{$a$}}}
    \Text(177,-207)[lb]{\Large{\Black{$d$}}}
    \Text(147,-189)[lb]{\Large{\Black{$b$}}}
    \Text(171,-191)[lb]{\Large{\Black{$c$}}}
  \end{picture}
}}
\hbox{ or } \hspace{0.6cm}
\mathcal{T}=\fcolorbox{white}{white}{
\scalebox{0.7}{
  \begin{picture}(56,38) (133,-202)
    \SetWidth{1.0}
    \SetColor{Black}
    \Vertex(145,-187){2}
    \Vertex(138,-200){2}
    \Vertex(153,-200){2}
    \Vertex(168,-200){2}
    \Line(138,-199)(145,-188)
    \Line(153,-200)(145,-187)
    \Text(139,-184)[lb]{\Large{\Black{$a$}}}
    \Text(128,-202)[lb]{\Large{\Black{$c$}}}
    \Text(156,-202)[lb]{\Large{\Black{$b$}}}
    \Text(176,-202)[lb]{\Large{\Black{$d$}}}
  \end{picture}
}}
$$
Conclusion: The set of Q-compatible topologies are the topologies such that $\{b\}$ and $\{c\}$ are equivalent and the first two topologies above, and the discrete topology.\\

- Let $\lhd:X\times X\longrightarrow X$ the trivial  quandle structure defined by
$M_Q=\begin{bmatrix}
a & a & a & a\\
b & b & b & b\\
c & c & c & c\\
d & d & d & d
\end{bmatrix},$ all topologies on $X$ are compatible with this quandle structure.\\

\begin{remark}
Let $Q=Q_1 \amalg Q_2 \amalg \cdot \cdot \cdot \amalg Q_n$ be a finite quandle which contains at most four elements, where the $Q_i$ are the orbits and let $\mathcal{T}=(Q, \le)$ be a topological space.
 We noticed that, if $\mathcal{T}$ is Q-compatible then for all $i\in [n]$, $T_{|Q_i}$ is coarse or discrete topology.
\end{remark}
Does this remark remain true for any finite quandles ? This is not the case. Indeed, let $\lhd:X\times X\longrightarrow X$ be the quandle structure defined by $M_Q=\begin{bmatrix}
a & a & a & a & a & a\\
b & b & b & b & b & b\\
d & e & c & c & c & c\\
c & f & d & d & d & d\\
f & c & e & e & e & e\\
e & d & f & f & f & f 
\end{bmatrix}$\\

- In the first step, we prove that Q is well defined. It is clear that the operation $\lhd$ satisfies the conditions (i) and (ii) of the definition of a quandle, moreover we have $R_c=R_d=R_e=R_f=Id$ and:
\begin{align*}
&R_a(c\lhd a)=c=d\lhd a=R_a(c)\lhd R_a(a)\qquad R_a(d\lhd a)=d=c\lhd a=R_a(d)\lhd R_a(a)\\
&R_a(e\lhd a)=e=f\lhd a=R_a(e)\lhd R_a(a)\qquad R_a(f\lhd a)=f=e\lhd a=R_a(f)\lhd R_a(a)\\
&R_a(c\lhd b)=f=d\lhd b=R_a(c)\lhd R_a(b)\qquad R_a(d\lhd b)=e=c\lhd b=R_a(d)\lhd R_a(b)\\
&R_a(e\lhd b)=d=f\lhd b=R_a(e)\lhd R_a(b)\qquad R_a(f\lhd b)=c=e\lhd b=R_a(f)\lhd R_a(f)\\
\end{align*}
and 
\begin{align*}
&R_b(c\lhd a)=f=e\lhd a=R_b(c)\lhd R_b(a)\qquad R_b(d\lhd a)=e=f\lhd a=R_ b(d)\lhd R_b(a)\\
&R_b(e\lhd a)=d=c\lhd a=R_b(e)\lhd R_b(a)\qquad R_b(f\lhd a)=c=d\lhd a=R_b(f)\lhd R_b(a)\\
&R_b(c\lhd b)=c=e\lhd b=R_b(c)\lhd R_b(b)\qquad R_b(d\lhd b)=d=f\lhd b=R_b(d)\lhd R_b(b)\\
&R_b(e\lhd b)=e=c\lhd b=R_b(e)\lhd R_b(b)\qquad R_b(f\lhd b)=f=d\lhd b=R_b(f)\lhd R_b(f)\\
\end{align*}
then $\lhd$ satisfies the condition (iii) of the definition of a quandle. So $(Q, \lhd)$ is a quandle.

- Secondly, if $$\mathcal{T}=\fcolorbox{white}{white}{
\scalebox{0.7}{
  \begin{picture}(101,33) (188,-172)
    \SetWidth{1.0}
    \SetColor{Black}
    \Arc(223,-161)(12.817,146,506)
    \Arc(253,-161)(12.817,146,506)
    \Vertex(190,-167){2}
    \Vertex(199,-167){2}
    \Vertex(218,-167){2}
    \Vertex(226,-167){2}
    \Vertex(249,-167){2}
    \Vertex(257,-167){2}
    \Text(187,-164)[lb]{\Large{\Black{$a$}}}
    \Text(197,-165)[lb]{\Large{\Black{$b$}}}
    \Text(216,-163)[lb]{\Large{\Black{$c$}}}
    \Text(224,-163)[lb]{\Large{\Black{$d$}}}
    \Text(246,-163)[lb]{\Large{\Black{$e$}}}
    \Text(253,-165)[lb]{\Large{\Black{$f$}}}
  \end{picture}
}}$$
we prove that $\mathcal{T}$ is a $Q$-compatible. We have $R_c=R_d=R_e=R_f=Id$ and $L_a(x)=a$ and $L_b(x)=b$ for all $x\in \{ a, b, c, d, e, f \}$, then it suffices to show that $R_a, R_b$ is an isomorphism and $L_c, L_d, L_e, L_f$ is a continuous maps.\\
We have $c\sim d$ and $e\sim f$, we obtain:
$$
R_a(c)=d\sim c=R_a(d), \hspace{2cm} R_b(c)=e\sim d=R_b(d),
$$
$$
R_a(e)=f\sim e=R_a(f), \hspace{2cm}  R_b(e)=c\sim d=R_b(f),
$$


then $R_a$ and $R_b$ is an isomorphism.\\
Moreover\\
$$
L_c(a)=d,\hspace{1cm} L_c(b)=e,\hspace{1cm} L_d(a)=c,\hspace{1cm} L_d(b)=f,\hspace{1cm} L_e(a)=f,
$$
$$
L_e(b)=c,\hspace{1cm} L_e(b)=c,\hspace{1cm} L_f(a)=e,\hspace{1cm} L_f(b)=d,
$$
and $L_c(x)=c, L_d(x)=d, L_e(x)=e$ and $L_f(x)=f$, for all $x\in \{c, d, e, f\}$.
Then $L_x$ is a continuous maps for all $x\in \{a, b, c, d, e, f\}$.\\
So $\mathcal{T}$ is $Q$-compatible.\\


\textbf{Acknowledgements}: The authors would like to thank Mohamed Elhamdadi for useful suggestions and comments.
	
	\textbf{Conflicts of interest}: none

\newpage

\end{document}